\documentclass[a4paper,12pt]{article}
\usepackage{amsmath, amssymb, amsthm}
\usepackage{hyperref}

\newcounter{theorem}
\newtheorem{thm}[theorem]{Theorem}
\newtheorem{lemma}[theorem]{Lemma}
\newtheorem{prop}[theorem]{Proposition}
\newtheorem{cor}[theorem]{Corollary}
\newtheorem{defn}[theorem]{Definition}

\newtheorem{question}[theorem]{Question}

\theoremstyle{remark}
\newtheorem*{remark*}{Remark}
\newtheorem{remark}[theorem]{Remark}
\newtheorem{example}[theorem]{Example}

\numberwithin{equation}{section}
\numberwithin{theorem}{section}

\newcommand{\triplename}{standard triple}

\newcommand{\labelledthing}[2]{\hspace{4pt}\buildrel {#2} \over #1 \hspace{3pt}} 

\newcommand{\labelledrightarrow}{\labelledthing{\longrightarrow}}
\newcommand{\iso}{\cong}

\newcommand{\R}{\mathbb{R}}
\newcommand{\Q}{\mathbb{Q}}
\newcommand{\Z}{\mathbb{Z}}

\renewcommand{\setminus}{\backslash}
\renewcommand{\emptyset}{\varnothing}
\newcommand{\dunion}{\amalg}
\newcommand{\tens}{\otimes}
\newcommand{\dsum}{\oplus}
\newcommand{\e}{\epsilon}

\DeclareMathOperator{\Span}{span}

\newcommand{\eb}{\partial_e}
\newcommand{\degen}[1]{\mathrm{Degen}_{#1}}

\newcommand{\alabel}{\label}

\begin{document}

\title{A classification of finite rank dimension groups by their representations in ordered real vector spaces}
\author{Greg Maloney and Aaron Tikuisis}

\maketitle

\begin{abstract}
This paper systematically studies finite rank dimension groups, as well as finite dimensional ordered real vector spaces with Riesz interpolation.
We provide an explicit description and classification of finite rank dimension groups, in the following sense.
We show that for each $n$, there are (up to isomorphism) finitely many ordered real vector spaces of dimension $n$ that have Riesz interpolation, and we give an explicit model for each of them in terms of combinatorial data.
We show that every finite rank dimension group can be realized as a subgroup of a finite dimensional ordered real vector space with Riesz interpolation via a canonical embedding.
We then characterize which of the subgroups of a finite dimensional ordered real vector space have interpolation (and are therefore dimension groups).
\end{abstract}

\section{Introduction}
Dimension groups are interesting algebraically, being only slightly more general than lattice-ordered abelian groups (in fact, the class of dimension groups comprises exactly the groups obtained by taking inductive limits of lattice-ordered abelian groups \cite[Theorem 3.21]{Goodearl:book}).
Moreover, dimension groups are particularly relevant in operator algebras, as they often appear as the $K_0$-group of a C$^*$-algebra, and thus constitute a crucial component in the range of the Elliott invariant \cite{Elliott:classificationAF, Elliott:classificationRR0, Elliott:classificationAT, Villadsen:AHrange}.
Dimension groups also appear in the study of topological Markov chains \cite{Krieger:DimFns,CuntzKrieger:algebras}, coding theory \cite{Farhane:coding}, and have examples in number theory \cite[Proposition 4.4]{EffrosHandelmanShen}.

Systematic study of dimension groups has been undertaken previously.
Most famously, Effros, Handelman and Shen proved in \cite[Theorem 2.2]{EffrosHandelmanShen} an equivalence between an abstract characterization of dimension groups (as ordered groups that are unperforated and have Riesz interpolation) and a concrete one (as inductive limits of ordered abelian groups of the form $(\mathbb{Z}^r, \mathbb{N}^r)$).
A more detailed exploration of the structure of finite rank dimension groups is contained in works by Effros and Shen \cite{EffrosShen:DimGroups, EffrosShen:Geometry}; Theorem 1.4 of the latter classifies all finite rank dimension groups under the restriction of being simple.
Goodearl generalized this classification to all simple dimension groups in \cite[Theorem 14.16]{Goodearl:book}.
Dimension groups that are also ordered real vector spaces (which we will refer to as ordered real vector spaces with Riesz interpolation) were studied systematically by Fuchs in \cite{Fuchs:InterpVSpaces, Fuchs:rieszspaces}.
Fuchs primarily focuses on describing the topological and order-theoretic structure of such vector spaces under the additional assumption of being anti-lattices.
Additionally, he shows that any ordered real vector space with interpolation can be embedded into a cartesian product of antilattices.
However, even with this result, there is much more to be understood about the structure of ordered real vector spaces with interpolation.
In particular, the results contained here do not seem to follow from Fuchs', and indeed, our techniques are quite different from his.

This paper concerns the question of describing all finite rank dimension groups and all finite dimensional ordered real vector spaces with Riesz interpolation.
Our ultimate results are explicit descriptions and classifications of finite rank dimension groups and finite dimensional ordered real vector spaces with Riesz interpolation.
For the finite dimensional ordered real vector spaces with Riesz interpolation, Theorem \ref{Rvspaces-Classification}, shows that for each $n$, there are (up to isomorphism) finitely many ordered real vector spaces of dimension $n$ that have Riesz interpolation, and we give an explicit model for each of them in terms of combinatorial data.
In Corollary \ref{Dimension-subgroup} and the following remarks, we see that every finite rank dimension group can be realized as a subgroup of a finite dimensional ordered real vector space with Riesz interpolation, via a canonical embedding with the property that every isomorphism of the dimension groups lifts to an isomorphism of the real vector spaces.
Theorem \ref{ConditionsThm} characterizes which of the subgroups of a finite dimensional ordered real vector space have interpolation (and are therefore dimension groups).
Putting these together yields an explicit description of all finite rank dimension groups (up to isomorphism), in terms of subgroups of $\R^n$ and combinatorial data describing the positive cone; additionally, an explicit description of when such dimension groups are isomorphic yields a complete classification of these dimension groups; altogether, this classification is the content of Corollary \ref{dgroups-Classification}.

To give an idea of the nature of the results, consider an $n$-dimensional ordered real vector space $V$ with Riesz interpolation that has $n$ extreme states (i.e.\ the positive functionals on $V$ separate its points).
In this case, each ideal $I$ of $V$ is determined by which of the extreme states are nonzero on $I$.
This induces a lattice isomorphism between the ideals of $V$ and a sublattice of the subsets of extreme states on $V$ (or, by labelling the states with numbers $1$ through $n$, a sublattice of $2^{\{1,\dots,n\}}$); $V$ is determined, as an ordered real vector space, by this sublattice of $2^{\{1,\dots,n\}}$, up to a permutation of $\{1,\dots,n\}$.

Note that the simplifying assumptions of the last example assure that the extreme states on $V$ already give a natural representation of $V$ into $\R^n$, and this representation provides a nice presentation of the ideal structure and positive cone of $V$.
In the case that $V$ has fewer than $n$ extreme states, Theorem \ref{RepresentationThm} is a non-trivial tool providing a representation of $V$ that still gives a nice presentation of the ideal structure and the positive cone.
Theorem \ref{RepresentationThm} also applies to finite rank dimension groups (and indeed, to more general ordered groups than these two disjoint classes).
The classification of finite rank dimension groups, however, is more complicated than that of finite dimensional ordered real vector spaces with Riesz interpolation, because certain obstructions to having Riesz interpolation are automatically avoided by ordered real vector spaces.

Section \ref{Prelim} contains the definitions of the basic concepts related to dimension groups and some standard background results.
In Section \ref{Reps}, we explore the representation of dimension groups into $\R^n$, leading the way to the main result, Theorem \ref{RepresentationThm}.
This result gives a canonical representation into $\R^n$ of many dimension groups (including those that have finite rank and those that are finite dimensional ordered real vector spaces), and describes the positive cone as a pull-back of a reasonably nice cone in $\R^n$.
Theorem \ref{RepresentationThm} paves the way for the canonical embedding used in the classification results.
In Section \ref{Necessary-sufficient} we solve the problem of which subgroups of $\R^n$ with the reasonably nice cones as in Theorem \ref{RepresentationThm} are in fact dimension groups.
The solution of this problem is Theorem \ref{ConditionsThm}, which gives necessary and sufficient conditions for such an ordered group to have Riesz interpolation.
Sections \ref{Rvspaces} and \ref{dgroups} contain the classifications of finite dimensional ordered real vector spaces with interpolation and finite rank dimension groups respectively; these classifications are given simply by collecting and specializing the results in Theorems \ref{RepresentationThm} and \ref{ConditionsThm}.
\section{Preliminaries\alabel{Prelim}}

\begin{defn}
An \textbf{ordered abelian group} (or simply ordered group) is an abelian group $G$ that is equipped with a partial order, $\leq$, satisfying the following:
\begin{enumerate}
\item[(i)] Compatibility with addition: for all $g, h, x\in G$, if $g\leq h$ then $g + x \leq h + x$; and 
\item[(ii)] Directedness: for all $g,h \in G$, there exists $y \in G$ such that
\[ \begin{array}{c} g \\ h \end{array} \leq y. \]
\end{enumerate}
An \textbf{ordered vector space} over a field $F \subseteq \R$ is a vector space $V$ with an ordering $\leq$ such that $(V,\leq)$ is an ordered group, and additionally satisfying:
\begin{enumerate}
\item[(iii)] Compatibility with scalar multiplication: for all $g, h \in G$ and $r \in F$, if $g \leq h$ and $r \geq 0$ then $rg \leq rh$.
\end{enumerate}
\end{defn}

One associates to an ordered abelian group its positive cone,
\[
G^+ := \{ g\in G \ : \ g \geq 0\}.
\]
This is a cone in $G$, meaning that it is closed under taking sums.  
It is a strict cone, meaning $G^+ \cap -G^+ = \{0\}$ (a consequence of $\leq$ being an order instead of a pre-order), and directedness of $(G,\leq)$ amounts to $G^+ - G^+ = G$.
Moreover, given any strict cone $C \subseteq G$ such that $C-C=G$, it makes $G$ into an ordered abelian group via the order $g \leq_C h$ iff $h-g \in C$.
Since the cone encapsulates all of the order information about $G$, there is a one-to-one correspondence between ordered abelian groups and groups with a strict cone $C$ satisfying $C-C=G$.
Hence, it is common practice (that will be taken here) to call $(G,G^+)$ an ordered abelian group (instead of $(G,\leq)$).

In the case that $V$ is a vector space over $F$ that is an ordered group with positive cone $V^+$, being an ordered $F$-vector space is equivalent to $V^+$ being closed under scalar multiplication by positive elements of $F$.
We call a cone $C$ in a real vector space $V$ a \textbf{real cone} if it is closed under scalar multiplication by positive real numbers.

Every ordered group is torsion-free; this is an easy consequence of the existence of a strict cone that generates the group.

Whenever we have an ordered group $G$ and a subgroup $H$, we will treat $H$ as an ordered group via the induced order, that is, $H^+ = G^+ \cap H$.

\begin{defn}
Let $G$ be an ordered abelian group.
An \textbf{ideal} of $G$ is a subgroup $H$ satisfying the following:
\begin{enumerate}
\item[(i)] (Order-)convexity: if $g,h \in H$ and $x \in G$ satisfies $g \leq x \leq h$ then $x \in H$.
\item[(ii)] $(H,H^+)$ is itself an ordered group.
(Since other conditions of an ordered group are automatic, this amounts to $H$ being a directed subset, or $H=H^+-H^+$).
\end{enumerate}
The ordered group $G$ is said to be \textbf{simple} if it contains no ideals other than $\{0\}$ and $G$.
\end{defn}

When $V$ is an ordered vector space, note that every ideal is automatically a subspace.
To see this, note that if $I$ is an ideal and $x \in I^+$, then for any scalar $r \in [0,1]$ we have $0 \leq rx \leq x$, and thus $rx \in I$.
It is then easy to see that $rx \in I$ for any scalar $r$ and $x \in I^+$, and by directedness, this holds in fact for any $x \in I$.

\begin{defn}
An \textbf{order unit} of an ordered group $G$ is an element $u\in G^+$ such that, for all $g\in G$, there exists $n\in \mathbb{N}$ such that $-n u \leq g \leq nu$.  
This is equivalent to saying that $u$ generates $G$ as an ideal.
\end{defn}

\begin{defn}
A \textbf{positive functional} on an ordered group $G$ is a group homomorphism $\phi:G \to \R$ that is positive, meaning that if $g\in G^+$, then $\phi (g) \geq 0$.  
If $u$ is an order unit for $G$ then a positive functional $\phi$ is a \textbf{state} (with respect to $u$) if $\phi(u) = 1$.
Let us denote by $S(G, u)$ the set of all states on $G$.  
When $u$ is understood, let us write this simply as $S(G)$.  
\end{defn}

The set $S(G,u)$ is a convex subset of the vector space of functionals on $G$, (see \cite[Proposition 6.2]{Goodearl:book}).
Let us denote by $\eb S(G,u)$ the set of all extreme points of $S(G,u)$.
In an important sense, the extreme points of $S(G,u)$ do not depend on the choice of order unit $u$.
Namely, if $u,u'$ are different order units then, of course, for every $\phi \in S(G,u)$, there is a unique real number $k_\phi$ such that $k_\phi\phi \in S(G,u')$.
The map $\phi \to k_\phi \phi$ is of course a bijection, and although it may not be affine, it sends $\eb S(G,u)$ to $\eb S(G,u')$ \cite[Proposition 6.17]{Goodearl:book}.

The following result shows that we have order units (and therefore states) for the ordered groups that we are mostly concerned with in this paper.

\begin{prop}
Let $(G,G^+)$ be an ordered group.
The ideals of $G$ satisfy the ascending chain condition if and only if every ideal of $G$ has an order unit.
\end{prop}

\begin{proof}
Note that for each element $x \in G^+$, $x$ is an order unit in the ideal generated by $x$.

$\Rightarrow$: Suppose that the ideals of $G$ satisfy the ascending chain condition.
Let $H$ be an ideal of $G$, and let $x_1 \in H^+$.
If $x_1$ is not an order unit for $H$, then let $H_1$ be the ideal generated by $x_1$, so that $H_1 \subsetneq H$, and there exists $x_2 \in H^+ \setminus H_1$.
Continuing this way, we either arrive at an order unit for $H$ or get an infinite, strictly increasing chain of ideals.

$\Leftarrow$: Suppose that every ideal of $G$ has an order unit.
Let
\[ H_1 \subseteq H_2 \subseteq \cdots \]
be an increasing chain of ideals.
Then the union $H = \bigcup H_i$ is an ideal, and therefore it has an order unit $x$.
Therefore, $x \in (H_i)^+$ for some $i$, from which it follows that $H \subseteq H_i$.
\end{proof}

\begin{defn}
Let $G$ be an ordered group.
\begin{enumerate}
\item[(i)]
$G$ is \textbf{unperforated} if, for any $g \in G$ and any positive integer $n$, if $ng \geq 0$ then it must be the case that $g \geq 0$.
\item[(ii)]
$G$ has \textbf{Riesz interpolation} (or just interpolation, for short) if, for all $a_1, a_2, b_1, b_2\in G$ satisfying
\[
\begin{array}{l}
a_1\\
a_2
\end{array} \leq \begin{array}{l}
b_1\\
b_2
\end{array}
\]
there exists some $z\in G$, called an \textbf{interpolant}, such that
\[
\begin{array}{l}
a_1\\
a_2
\end{array} \leq z \leq \begin{array}{l}
b_1\\
b_2
\end{array}.
\]
\item[(iii)]
$G$ is a \textbf{dimension group} if it is unperforated and has Riesz interpolation.
\end{enumerate}
\end{defn}
We will use (often implicitly) some important facts proven in \cite{Fuchs:RieszGrps} about the ideals of a dimension group.

\begin{prop}
Let $G$ be a dimension group.
\begin{enumerate}
\item[(i)] \cite[Theorem 5.6]{Fuchs:RieszGrps}
The ideals of $G$ form a distributive lattice, with operations $(I,J) \mapsto I+J$ and $(I,J) \mapsto I \cap J$.
\item[(ii)] \cite[Proposition 5.8]{Fuchs:RieszGrps}
For ideals $I_1,I_2$ of $G$, we have $(I_1+I_2)^+ = I_1^+ + I_2^+.$
\end{enumerate}
\end{prop}

Let $G$ be a dimension group and let $I$ be an ideal of $G$.
Consider the intersection of all of the kernels of positive functionals $\phi$ defined on an ideal $J$ containing $I$ for which $\phi(I) = 0$, that is
\begin{equation}
K_I = \bigcup \{\ker \phi: \phi:J \to \R \textrm{ is positive, }J\textrm{ is an ideal, }I \subseteq J \textrm{ and } \phi(I) = 0\}.
\alabel{KI-definition}
\end{equation}
Evidently, $K_I$ contains $I$; however, they may not be equal.
Their inequality represents a sort of degeneracy in the positive cone, as is best illustrated in the case that $G$ is a finite dimensional real ordered vector space $V$ that is simple.

In this case, the number of extreme states on $V$ is equal to the dimension of $V$ exactly when the closure of $V^+$ contains no non-trivial subspace (in fact, the difference between $\dim V$ and $|\eb S(V)|$ is exactly the dimension of the largest subspace contained in $\overline{V^+}$).
The largest subspace of the closure of $V^+$ can in fact be found by taking the intersection of the kernels of the extreme states on $V$ (or equivalently, the kernels of all positive functionals on $V$).

In the non-simple case, there may be positive functionals on some ideal that do not extend as positive functionals on the entire ordered group, which is why we use positive functionals on ideals $J$ containing $I$ in \eqref{KI-definition}.
We define the \textbf{degeneracy quotient} of $I$ as
\[ \degen{I} := K_I/I. \]

For an ordered group $G$ and a field $F$ contained in $\R$, we may form the tensor product $G \tens_{\Z} F$, into which $G$ naturally embeds.
Using the smallest (vector space) cone in $G \tens_{\Z} F$ that contains the image of $G^+$ makes $G \tens_{\Z} F$ an ordered vector space, and the embedding $G \to G \tens_{\Z} F$ is an order embedding precisely if $G$ is unperforated.

\begin{prop}\alabel{Qtensor}
Let $G$ be an unperforated ordered group.
\begin{enumerate}
\item[(i)] There is a one-to-one correspondence between the ideals of $G$ and those of $G \tens_\Z \Q$ given by $I \mapsto I \tens_{\Z} \Q$.
\item[(ii)] If $u \in G$ is an order unit then $u \tens 1$ is an order unit for $G \tens_\Z \Q$ and $S(G,u) = S(G \tens_\Z \Q, u \tens 1)$.
\item[(iii)] If $G$ has Riesz interpolation then so does $G \tens_\Z \Q$.
\end{enumerate}
\end{prop}

\begin{proof}
(i) The proof of the last part of the statement of \cite[Lemma 2.1]{EffrosShen:Geometry} works here; (ii) and (iii) are easy.
\end{proof}

\section{Representation using functionals\alabel{Reps}}

The main result of this section is Theorem \ref{RepresentationThm}, which describes how to represent certain dimension groups in $\R^n$, in such a way that the ideals and the positive cone of the dimension group are described using the embedding and some combinatorial data.  
The next example shows that, to find such a representation, it is not enough simply to tensor a dimension group $G$ with $\R$.

\begin{example}\alabel{TensorEx}
Let $\theta\in \mathbb{R}$ be irrational and let $G = \mathbb{Z}^2$ equipped with the positive cone
\[
G^+ := \{ 0\} \cup \{ (a,b) \in G \ : \ a + \theta b > 0\}.
\]
Then $(G, G^+)$ is a dimension group \cite{EffrosShen:ContinuedFractions}.

Tensoring $G$ with $\R$ yields a two-dimensional order real vector space, namely $\R^2$ with positive cone
\[
\{ 0\} \cup \{ (a,b) \in \R^2 \ : \ a + \theta b > 0\}.
\]
The map $G \to \R$ given by $(a, b) \mapsto a + \theta b$ is a positive embedding of $G$ into $\R$, and is, in fact, the embedding that we will produce in Theorem \ref{RepresentationThm}.  
The embedding of $G$ into $\R^2$ is undesirable because it introduces degeneracy into $\R^2$ that was not present in the original group $G$.  
In particular, the trivial ideal $\{ 0\}$ of $G$ is the kernel of the unique state on $G$, but within $\R^2$ the kernel of the unique state is non-trivial.
\end{example}

When degeneracy is present in the group $G$ (i.e., when $\degen{I} \neq 0$ for some $I$), additional data (an embedding of $\degen{I}$ into a real vector space) is requested by the following theorem to produce the embedding of $G$.

\begin{thm}\alabel{RepresentationThm}
Let $G$ be a dimension group that has finitely many ideals, finitely many extreme states on each ideal, and such that, for every ideal $I$, there exists an embedding $\psi_I:\degen{I} \to \R^{k_I}$ (for some $k_I$), for which $\psi_I(\degen{I})$ spans $\R^{k_I}$.
Then there exists a map $\phi = (\phi_1,\dots,\phi_n):G \to \R^n$, a sublattice $\mathcal{S}$ of $2^{\{1,\dots,n\}}$ and, for each $S \in \mathcal{S}$, a subset $P^>_S$ of $S$ such that:
\begin{enumerate}
\item[(i)] $\phi$ is one-to-one and $\phi(G)$ spans $\R^n$.
\item[(ii)] The positive cone is given by
\[ G^+ = \bigcup_{S \in \mathcal{S}} \{g \in G: \phi_i(x) > 0\ \forall i \in P^>_S, \phi_i(x)=0\ \forall i\not\in S\}. \]
\item[(iii)] There is a lattice isomorphism between $\mathcal{S}$ and the ideals of $G$, given by $S \mapsto I_S$ where for $S \in \mathcal{S}$,
\[ I_S = \{g \in G: \phi_i(g) = 0\ \forall i \not\in S\}. \]
\item[(iv)] The order units of the ideal $I_S$ are exactly all $g \in I_S$ for which
\[ \phi_i(g) > 0 \ \forall  i \in P^>_S. \]
\item[(v)] For each $S \in \mathcal{S}$, $\phi(I_S)$ spans
\[ \{(x_1,\dots,x_n) \in \R^n: x_i = 0\ \forall i \not\in S\} \]
and the real cone generated by the image under $\phi$ of the order units of $I_S$ is
\[ \{(x_1,\dots,x_n) \in \R^n: x_i = 0\ \forall i \not\in S, x_i > 0\ \forall i \in P^>_S\}. \]
\item[(vi)] For each ideal $I$ of $G$ which is not the proper intersection of two ideals, let $S \in \mathcal{S}$ be such that $I = I_S$, and set
\[ D_S := \bigcup_{T \supset S} P^>_T \cup S. \]
Then $|D_S| = n-k_I$, and upon identifying $\R^{k_{I}}$ with
\[ \{(x_1,\dots,x_n): x_i = 0\ \forall i \in D_S\}, \]
(by sending the $i^{\text{th}}$ coordinate to the $i^{\text{th}}$ nonzero coordinate), the following commutes
\[
\begin{array}{rcl}
\degen{I} & \to & G/I \\ 
\scriptsize{\psi_{I}} \downarrow\ \  & & \downarrow\ \  \scriptsize{\alpha} \\
\R^{k_{I}} & = & \{(x_1,\dots,x_n): x_i = 0\ \forall i \in D_S\}
\end{array}
\]
where $\alpha(g) = (x_1,\dots,x_n)$ with $x_i = 0$ if $i \in D_S$ and $x_i = \phi_i(g)$ otherwise.
\end{enumerate}
Moreover, the number $n$, the lattice $\mathcal{S}$, and the subsets $P^>_S$ are determined (up to a permutation of the indices $\{1,\dots,n\}$) by the lattice of ideals of $G$, the space of positive functionals on each ideal, the restriction maps between the spaces of positive functionals, and the value of $k_i$.
The embedding $\phi$ is determined (up to a permutation of the indices $\{1,\dots,n\}$) by the ordered group $G$ and the embeddings $\psi_I$.
\end{thm}

\begin{remark}
This last result allows us to represent in $\R^n$ any countable dimension group with finitely many ideals and finitely many extreme states on each ideal---since in that case, the rank of $\degen{I}$ is at most $\aleph_0$, so that we can in fact find an embedding $\psi_I:\degen{I} \to \R$.
However, if $\psi_I$ is allowed to be chosen in such an arbitrary manner, then the last theorem can produce different representations $\phi,\phi':G \to \R^n$ such that for no vector space isomorphism $\alpha:\R^n \to \R^n$ does the following commute
\[ \begin{array}{rcl}
G & \labelledrightarrow{\phi} & \R^n \\
= & & \downarrow \scriptsize{\alpha} \\
G & \labelledrightarrow{\phi'} & \R^n.
\end{array} \]

In the following situations, we can pick $\psi_I$ canonically, so that different representations given by Theorem \ref{RepresentationThm} do lift to vector space isomorphisms of $\R^n$.
\begin{enumerate}
\item[(i)] If $G$ is a finite dimensional ordered vector space over $\R$ then $\degen{I}$ is itself a finite dimensional vector space, so we can let $\psi_I$ be a vector space isomorphism.
\item[(ii)] If $G$ is a finite dimensional ordered vector space over $\Q$ then $\degen{I}$ is itself a finite dimensional vector space over $\Q$, so we can let $\psi_I$ be the embedding $\degen{I} \to \degen{I} \tens \R$.
\item[(iii)] More generally, if $\degen{I}$ has finite rank then $\degen{I} \tens \R$ is a finite dimensional real vector space, and again we let $\psi_I$ be the embedding $\degen{I} \to \degen{I} \tens \R$.
\end{enumerate}
\end{remark}

A proof of Theorem \ref{RepresentationThm} comes by close examination of the states on the ideals of $G$.

For ideals $I \subseteq J$, we have a restriction map $r$ from the positive functionals on $J$ to those on $I$.
This restriction map sends each state on $I$ either to zero or to a scalar multiple of a unique state on $J$.
Modulo this adjustment by a scalar multiple, the upcoming results may be summarised as follows.

\begin{itemize}
\item Lemma \ref{StateRestriction}: $r$ sends $\eb S(I)$ to $\eb S(J) \cup \{0\}$.
\item Lemma \ref{RestrictionOneToOne}: $r$ is one-to-one on the set of extreme states that don't map to zero.
\item Lemma \ref{StatePullBack} can be summarized as the existence of a pull-back state in the following diagram (where every map is a restriction map)
\[
\begin{array}{rrcll}
 & \exists \tau \in & \eb S(I_1 + I_2) & & \\
 & \swarrow & & \searrow & \\
\tau_1 \in \eb S(I_1) & & & & \eb S(I_2) \ni \tau_2 \\
 & \searrow & & \swarrow & \\
 & \tau_2 \in & \eb S(I_1 \cap I_2) & & \\
\end{array}
\]
\end{itemize}

\begin{lemma}\alabel{StateRestriction}
Let $G$ be an ordered group with Riesz interpolation and $I$ an ideal.
Suppose that $G,I$ have order units $u,v$ respectively.
Then for $\tau \in \eb S(G,u)$, $\tau|_I$ is a scalar multiple (possibly $0$) of some $\tau' \in \eb S(I,v)$.
\end{lemma}

\begin{proof}
Replacing $G$ by $G \tens \Q$, by Proposition \ref{Qtensor}, we may assume that $G$ is an ordered $\Q$-vector space with Riesz interpolation.

Suppose for a contradiction that $\tau|_I = f_1 + f_2$ where $f_1,f_2$ are linearly independent positive functionals on $I$.
So, there exist $x,y \in I^+$ such that
\[ f_1(x) < f_2(x) \text{ and } f_1(y) > f_2(y). \]

Let $\e > 0$ be such that
\[ f_1(x) < f_2(x) - \e \text{ and } f_2(y) < f_1(y) - \e. \]
By \cite[Proposition I.9.1]{GoodearlHandelmanLawrence}, let $b \in G$ be such that
\begin{equation} \min \{\tau(x), \tau(y)\} - \e < \tau(b) \alabel{StateRestriction-tau-b} \end{equation}
and
\[ \tau'(b) < \begin{array}{c} \tau'(x) \\ \tau'(y) \end{array} \]
for all $\tau' \in \eb S(G,u)$.
Consequently, we have
\[ \begin{array}{c} 0 \\ b \end{array} \leq \begin{array}{c} x \\ y, \end{array} \]
so using interpolation there exists $z \in G$ such that
\[ \begin{array}{c} 0 \\ b \end{array} \leq z \leq \begin{array}{c} x \\ y. \end{array} \]

Since $0 \leq z \leq x,y$, we must have $z \in I^+$ and
\[ f_i(z) \leq \begin{array}{c} f_i(x) \\ f_i(y) \end{array} \]
for $i=1,2$.

But then,
\begin{align*}
\tau(b) &\leq \tau(z) \\
&= (f_1+f_2)(z) \\
&\leq f_1(x) + f_2(y) \\
&\leq f_1(x) + f_2(x) - \e \\
&= \tau(x) - \e,
\end{align*}
and likewise,
\[ \tau(b) \leq \tau(y) - \e. \]
This contradicts \eqref{StateRestriction-tau-b}.
\end{proof}

\begin{lemma}\alabel{RestrictionOneToOne}
Let $G$ be an ordered group with Riesz interpolation and $I$ an ideal.
Suppose that $G$ has an order unit $u$.
Let $\tau_1,\tau_2 \in \eb S(G,u)$ be such that $\tau_1 \neq \tau_2$.
If $\tau_2|_I$ is a scalar multiple of $\tau_1|_I$ then the $\tau_2|_I = 0$.
\end{lemma}

\begin{proof}
Again, we may assume that $G$ is an ordered $\Q$-vector space with interpolation.
Suppose for a contradiction that $\tau_1|_I=c\tau_2|_I$, $c\in(0,\infty)$.
Let $x\in I^+$ be such that $\tau_1(x)> c/2$.
Notice that if $0 \leq z \leq x$ then $z \in I^+$ and so $\tau_1(z)=c\tau_2(z)$.
Since $\tau_1\neq \tau_2$, we can use \cite{Edwards} and \cite[Theorem I.9.1]{GoodearlHandelmanLawrence} to find $y \in G$ such that $\tau_1(y) \in (0,1/2)$ and $\tau_2(y) > c/2$, while $\tau(y)>0$ for all $\tau \in \eb S(G,u)$.
Likewise, we can find $b \in G$ such that $\tau_2(b) \in (c/2,\tau_2(y))$ and
\[ \tau(b) < \begin{array}{c} \tau(y) \\ \tau(x) \end{array} \]
for all $\tau \in \eb S(G,u)$.
Hence, we have
\[ \begin{array}{c} 0 \\ b \end{array} \leq \begin{array}{c} x \\ y \end{array}, \]
but if $z$ were an interpolant then
\[ \frac{c}2 < \tau_2(b) \leq \tau_2(z) = c\tau_1(z) < c\tau_1(y) < \frac{c}2. \]
Hence, this contradicts the fact that $G$ has Riesz interpolation.
\end{proof}

\begin{lemma}\alabel{StatePullBack}
Let $G$ be an ordered group with Riesz interpolation and $I_1,I_2$ be ideals such that $I_1+I_2=G$.
Suppose that $G,I_1,I_2$ have order units $u,v_1,v_2$ respectively.
Suppose that we have $\tau_i \in \eb S(I_i,v_i)$ for $i=1,2$, and that on $I_1 \cap I_2$, the $\tau_i$ are nonzero scalar multiples of each other.
Then there exists $\tau \in \eb S(G,u)$ such that for $i=1,2$, $\tau_i = c_i\tau|_{I_i}$ for some scalar $c_i$.
\end{lemma}

\begin{proof}
We can define $f:G \to \R$ such that $f|_{I_i} = d_i\tau_i$ for some scalars $d_i > 0$.
Then $f$ is a positive functional on $G$.
To show that $f$ is a scalar multiple of an extreme state, suppose that
\[ f = g+h \]
where $g,h$ are positive functionals on $G$, and let us show that $g,h$ are linearly dependent.

Since $f|_{I_1 \cap I_2} \neq 0$, WLOG, we have $g|_{I_1 \cap I_2} \neq 0$.
Then,
\[ d_1\tau_1 = f|_{I_1} = g|_{I_1} + h|_{I_1}, \]
and since $t_1$ is an extreme state, this implies that $g|_{I_1},h|_{I_1}$ are linearly dependent, so let
\[ h|_{I_1} = Kg|_{I_1} \]
for some scalar $K$.
Likewise, we see that $g|_{I_2},h|_{I_2}$ are linearly dependent, and since
\[ h|_{I_1 \cap I_2} = Kg|_{I_1 \cap I_2} \]
and $g|_{I_1 \cap I_2} \neq 0$, we must have that
\[ h|_{I_2} = Kg|_{I_2}. \]
Thus, $h=Kg$, so $g,h$ are linearly dependent, as required.
\end{proof}

\begin{cor}\alabel{GlobalStates}
Let $G$ be an ordered group with Riesz interpolation, such that every ideal $I$ of $G$ has an order unit $v$ and $\eb S(I,v)$ is finite.
For any ideal $I$ with order unit $v$ and any $\tau \in \eb S(I,v)$, there exists some ideal $\tilde{I}$ with order unit $w$ and some $\tilde{\tau} \in \eb S(\tilde{I},w)$ such that $\tau$ is a scalar multiple of $\tilde{\tau}|_I$, and $\tilde{I}$ is maximal in the sense that for any ideal $J$ satisfying $I \subset J \not\subseteq \tilde{I}$, if $f:J \to \R$ satisfies $f|_I = \tau$ then $f$ is not a positive functional on $J$.
\end{cor}

\begin{proof}
Let $\tilde{I}$ be the sum of all ideals $J$ such that $\tau$ extends to $J$ (as a scalar multiple of an extreme state on $J$).
If $\tau$ extends to a scalar multiple of an extreme state on the ideal $J$ then, by Lemma \ref{RestrictionOneToOne}, this extension is unique.
By Lemma \ref{StatePullBack}, we can see that if $I_1$ and $I_2$ are both ideals, and $\tau$ extends to a scalar multiple of an extreme state on each of $I_1$ and $I_2$ then $\tau$ extends to a scalar multiple of an extreme state on $I_1 + I_2$.
Thus, we see that $\tau$ extends to a scalar multiple of an extreme state on all of $\tilde{I}$.

Obviously, $\tilde{I}$ is maximal in the sense that it contains every ideal upon which $\tau$ has an extension that is a scalar multiple of an extreme state.
However, if $\tau$ extends to a positive functional $f$ on an ideal $J \supset I$ then since $\eb S(J)$ is finite, $f$ can be written as a linear combination (with positive coefficients) of extreme traces on $J$, and it follows that one of the extreme traces appearing in the linear combination must restrict to a nonzero scalar multiple of $\tau$.
Thus, $\tilde{I}$ is also maximal in the sense stated.
\end{proof}

From Corollary \ref{GlobalStates}, we that when $G$ is a dimension group as in Theorem \ref{RepresentationThm}, we can define functionals $\tau_1, \dots, \tau_k:G \to \R$ such that for every ideal $I \subseteq G$ and every $\tau \in \eb S(I)$, there exists a unique $i$ such that $\tau$ is a scalar multiple of $\tau_i$.
For each $I$, $\tau_i|_I$ is positive if and only if $\tau_i$ is a scalar multiple of an extreme state on $I$.
For each $i$, there is a maximal ideal upon which $\tau_i$ is positive; we denote this ideal by
\[ I_{\tau_i}. \]

In order to achieve condition (iii) of Theorem \ref{RepresentationThm}, we may need to include more functionals than $\tau_1,\dots,\tau_k$ in the list $\phi_1,\dots,\phi_n$.
For example, consider the case that $G = \mathbb{R}^2$ with the order given by
\[ G^+ = \{(0,0)\} \cup \{(x,y) \in \R^2: x > 0\}. \]
In this case, the ordered group (an ordered real vector space) is simple, and it has only one extreme state (namely, the functional $x$).

For each ideal $I$ of $G$, $\degen{I}$ represents exactly the difference between the ideal $I$ and what we would get if we used only $\tau_1,\dots,\tau_k$ to pick out the ideal.
Thus, we shall use $\psi_I$ to give us more functionals which exactly recover each ideal.

If an ideal $I$ can be written as the proper intersection of two ideals $J_1$ and $J_2$, and if we can write each of $J_1$ and $J_2$ as an intersection of kernels of some of the functionals $\phi_1,\dots,\phi_n$ then of course we can do the same for $I$.
This is why we only need to use the embeddings $\psi_I$ for ideals $I$ which are not the proper intersection of two ideals.
The next result characterizes such ideals.

\begin{lemma}\alabel{MinSuperIdeal}
Suppose that the ideals of $G$ satisfy the descending chain condition.
Let $I$ be a proper ideal of $G$.
Then either $I$ is the proper intersection of two ideals (i.e.\ $I = J_1 \cap J_2$, $J_1 \neq I \neq J_2$) or there exists a unique ideal $I'$ that properly contains $I$ and that is minimal with respect to this condition (i.e.\ if $I \subsetneq J$ then $I' \subseteq J$).
\end{lemma}

\begin{proof}
Let $I'$ be the intersection of all the ideals that properly contain $I$.
If $I' \neq I$ then it is clearly the minimum ideal that properly contains $I$.
Otherwise, $I' = I$ and since the ideals satisfy the descending chain condition, $I$ is the intersection of finitely many ideals which properly contain $I$.
This implies that $I$ is the proper intersection of two ideals.
\end{proof}

Whenever we have the set up of the last lemma, and the ideal $I$ is not the proper intersection of two ideals, we will continue to use the notation $I'$.

\begin{lemma}
Suppose that the ideals of $G$ satisfy the descending chain condition.
Let $I$ be a proper ideal of $G$ that is not the proper intersection of two ideals.
Let $f:G \to \R$ be a functional satisfying $f(I) = 0$ and $f(I') \neq 0$.
If $J \subset G$ is another ideal satisfying $f(J) = 0$ then $J \subseteq I$.
\end{lemma}

\begin{proof}
If $J \not\subseteq I$ then $I \subsetneq I+J$, and so $I' \subseteq I+J$.
However, if $f(J)=0$ then $f(I+J)=0$, in contradiction to $f(I') \neq 0$.
\end{proof}

\begin{lemma}\alabel{zero-functional-technical-lemma}
Suppose that the ideals of $G$ satisfy the descending chain condition.
Let $J$ be a proper ideal of $G$ that is not the proper intersection of two ideals.
Let $f:G \to \R$ be a functional such that $f(J) = 0$ and $f(J') \neq 0$.
Let $I$ be an ideal such that $f(I) \neq 0$ but $f(K) = 0$ for all $K \subsetneq I$.
Then:
\begin{enumerate}
\item[(i)] $I$ has a unique maximal subideal, namely $I \cap J$.
\item[(ii)] Either $I=J'$ or we have $I\cap J \subsetneq J$ and $I+J=J'$.
\item[(iii)] If $I \neq J'$ then every proper subideal of $J'$ is a subideal of either $I$ or $J$.
\end{enumerate}
\end{lemma}

\begin{proof}
(i) Let $K \subsetneq I$.
Then $f(K)=0$ so by the last lemma, $K \subseteq J$.
Hence every ideal of $I$ is contained in $I \cap J$.
Moreover, $I \cap J \neq I$, since $f$ vanishes on $I \cap J$ but not on $I$.

(ii) We have $J \subsetneq I+J$ and so $J' \subseteq I+J$.
Consider the case that $I\neq J'$.
If $I\cap J = J$ then $J \subseteq I$ so $J' \subsetneq I$.
But then $f$ does not vanish on the proper ideal $J'$ of $I$, a contradiction, so that $I\cap J \subsetneq J$.
If $I+J \neq J'$ then $I \not\subseteq J'$ and so $I\cap J' \subsetneq I$.
It follows from (i) that $I \cap J' = I \cap J$, and by using distributivity,
\[ J = (J'\cap J) + (I \cap J) = (J'\cap J) + (J'\cap I) = J' \cap (J+I) = J', \]
a contradiction.

(iii) Suppose that $I \neq J'$ and $K$ is a proper subideal of $J'$ such that $K \not\subseteq I$ and $K \not\subseteq J$.
We have first that $J \subsetneq J+K$ and so $J+K=J'$.
Secondly, $K\cap I \subsetneq I$ and so $I\cap K \subseteq I \cap J$ by (i).
Using distributivity, we have
\[ I \cap J = (I \cap J) + (I \cap K) = I \cap (J + K) = I \cap J', \]
and by (ii) we know that $I \cap J' = I \cap (I+J) = I$, which yields $I \cap J = I$, a contradiction to (ii).
\end{proof}

We have already described how to get functionals $\tau_1,\dots,\tau_k$ that induce every extreme state on every ideal.
Let us pick out the ideals of $G$ that are not the proper intersection of two ideals, and label them
\[ J_1,\dots, J_\ell. \]
For $s=1,\dots,\ell$, we use $\psi_{J_s}:\degen{J_s} \to \R^{a_s}$ (where $a_s = k_{J_s}$) to define functionals $f_{s,1},\dots,f_{s,a_s}$ from $G$ to $\R$.
Therefore, we have
\begin{equation} J_s = \bigcap \{\ker \tau_j: \tau_j|_{J_s'} \geq 0 \text{ and } \tau_j(J_s) = 0\} \cap \ker f_{s,1} \cap \cdots \cap \ker f_{s,a_s}. \alabel{J-intersectioneqn} \end{equation}

\begin{prop}\alabel{linindep}
The functionals 
\[ \tau_1, \dots, \tau_k, f_{1,1},\dots,f_{1,a_1},\dots,f_{\ell,1},\dots,f_{\ell,a_\ell} \]
are linearly independent over $\R$.
\end{prop}

\begin{proof}
Suppose that
\begin{equation} c_1 \tau_1 + \cdots + c_k \tau_k + d_{1,1}f_{1,1} + \cdots + d_{\ell,a_\ell} f_{\ell,a_\ell} = 0 \alabel{homogeqn} \end{equation}
for some scalars $c_1, \dots, c_k, d_{1,1}, \dots, d_{\ell,a_\ell} \in \R$.

We will show, by induction, that for the ideal $I \subseteq G$, we have
\begin{align*}
\tau_i|_I \neq 0 &\Rightarrow c_i = 0, \text{ and} \\
f_{s,t}|_I \neq 0 &\Rightarrow d_{s,t} = 0.
\end{align*}
Since $\tau_i|_G \neq 0$ and $f_{s,t}|_G \neq 0$ for all $i,s,t$, this inductive argument will prove linear independence.

We need to be careful about the order in which we enumerate the ideals for the induction process.
We want to enumerate them such that when doing the step for an ideal $I$, we may assume the induction hypothesis for all proper subideals of $I$; additionally, if $I=J_s$ and another ideal $J$ is the proper intersection of two ideals and is maximal in $J_s'$, then we may assume that the induction hypothesis holds for $J$.
Some explanation is required for why this is possible.

Let $\prec'$ denote the relation on the ideals of $G$ given by $I \prec' J_s$ if $I$ is the proper intersection of two ideals and $I$ is a maximal subideal of $J_s'$; $\prec'$ is clearly a (strict) pre-order.
Let $\prec$ denote the (strict) pre-order relation generated by $\prec'$ and $\subsetneq$.
We want to enumerate in a non-decreasing order with respect to $\prec$, and to show that this is possible, we need the following.

\begin{claim}
The pre-order $\prec$ is antisymmetric, and therefore a partial order.
\end{claim}

\begin{proof}[Proof of claim.]
Notice that if $I \prec' J$ and $J \subsetneq K$ then since $J'$ is the minimum superideal of $J$, we must have $J' \subseteq K$.
Since $I \prec' J$ implies that $I \subsetneq J'$, we must have $I \subsetneq K$.

It follows that if $I \prec J$ then either $I \subsetneq J$ or $I \subseteq K \prec' J$ for some ideal $K$.
In either case, it cannot happen that $I = J$.

Now, if $I \preccurlyeq J$ and $J \preccurlyeq I$ then either $I = J$ or else $I \prec J \prec I$, in which case $I \prec I$, which contradicts what was just shown.
\end{proof}

Let us now do the inductive step, where we may assume that for any ideal $J \prec I$, the inductive hypothesis holds for $J$.

For each $i$, if $\tau_i|_I = 0$ then there is nothing to prove; if $\tau_i|_K \neq 0$ for some $K \subsetneq I$ then $c_i = 0$ by the induction hypothesis.
The only case of interest, then, is that $\tau_i|_I \neq 0$ but $\tau_i|_K = 0$ for all $K \subsetneq I$.
Recall that $I_{\tau_i}$ is the maximal ideal upon which $\tau_i$ is positive.
If $I \not\subseteq I_{\tau_i}$ then $I \cap I_{\tau_i} \subsetneq I$ and so $\tau_i|_{I_{\tau_i} \cap I} = 0$ by induction.
Hence, there exists $g:G \to \R$ such that $g|_I = 0$ but $g|_{I_{\tau_i}} = \tau_i|_{I_{\tau_i}}$.
Then $I_{\tau_i} \neq I+I_{\tau_i}$ but $g \geq 0$ on $I+I_{\tau_i}$ (since $g = 0$ on $I$ and $g \geq 0$ on $I_{\tau_i}$), which contradicts Corollary \ref{GlobalStates}.
Hence, we must have $I \subseteq I_{\tau_i}$, and thus $\tau_i|_I$ is a (nonzero) scalar multiple of some $\tau \in \eb S(I)$.

Likewise, we have a trichotomy for each $f_{s,t}$: the only case where there is something to prove is the case that $f_{s,t}|_I \neq 0$ but $f_{s,t}|_K=0$ for all $K \subsetneq I$.
For this case, we look to Lemma \ref{zero-functional-technical-lemma}: by (i), $I \cap J_s$ is the unique maximal subideal of $I$.
By (ii), we can break our analysis into two cases, and proceed to show in each case that $d_{s,t}=0$ for all $t=1,\dots,a_s$.
\begin{enumerate}
\item[Case 1.] $I = J_s'$. \\
In this case, we cannot have that $f_{s',t'}|_I \neq 0$ but $f_{s',t'}|_K = 0$ for all $K \subsetneq I$, for some $s' \neq s$.
Certainly, if we did have this, then $J_{s'} \cap I$ would be the unique maximal subideal of $I$, and so $J_s = J_{s'} \cap I$, expressing $J_s$ as the proper intersection of ideals.

Restricting \eqref{homogeqn} to $\bigcap \{\ker \tau_i: \tau_i(J_s) = 0\}$ gives
\[ \left(d_{s,1} f_{s,1} + \cdots + d_{s,a_s} f_{s,a_s}\right) \big|_{\bigcap \{\ker \tau_i: \tau_i|_{J_s'} \geq 0 \text{ and } \tau_i(J_i)=0\}} = 0. \]
If the scalars $d_{s,t}$ are not all nonzero, then this amounts to a linear dependence of the components of $\psi_{J_s}$.
However, since the range of $\psi_{J_s}$ spans $\R^{a_s}$, the components are linearly dependent, so $d_{s,t} = 0$ for all $t$.

\item[Case 2.] $I\neq J_s'$. \\
In this case, by Lemma \ref{zero-functional-technical-lemma} (iii), it follows that either $J_s \prec I$ or else $I = J_{s'}$ for some $s'$.

Let us first handle the case that $I=J_{s'}$ for some $s'$.
Assume (by reordering indices) that $\tau_1, \dots, \tau_b$ are nonzero on $J_s$, $\tau_{b+1},\dots,\tau_{b'}$ are nonzero on $J_{s'}$, $\tau_1,\dots,\tau_{b'}$ are all zero on $J_{s'}\cap J_{s}$ (and therefore on any proper subideal of either $J_s$ or $J_{s'}$), and for $i=b+1,\dots,k$, either $\tau_i|_{J_s'} =0$ or $c_i=0$.

Let us first show that $\tau_1|_{J_{s'}} = 0$.
As argued above, since $\tau_1|_K = 0$ for all $K \subsetneq J_s$, we have $\tau_1|_{J_s} \geq 0$.
Likewise, we have $\tau_1|_{J_{s'}} \geq 0$.
Consequently, $\tau_1|_{J_s'} \geq 0$, yet it is nonzero, so it is a scalar multiple of an extreme state.

We may define $f:J'_s \to \R$ by $f|_{J_s} = 0$ and $f|_{J_{s'}} = \tau_1|_{J_{s'}}$.
Likewise, we may define $g:J'_s \to \R$ by $g|_{J_s'} = \tau_1|_{J_{s'}}$ and $g|_{J_s} = 0$.
Since $J_s' = J_s + J_{s'}$, we have $f+g = \tau_1|_{J_s'}$.
If $\tau_1|_{J_{s'}} \neq 0$ then this would contradict the fact that $\tau_1|_{J_s'}$ is a scalar multiple of an extreme state.
Therefore, $\tau_1|_{J_{s'}} = 0$.
Likewise, $\tau_i|_{J_{s'}} = 0$ for $i=1,\dots,b$ and $\tau_i|_{J_s} = 0$ for $i=b+1,\dots,b'$.

The equation \eqref{homogeqn}, restricted to $J_s'$, becomes
\begin{align*}
&\left(c_1\tau_1 + \cdots + c_b\tau_b + d_{s',1} f_{s',1} + \cdots + d_{s',a_{s'}} f_{s',a_{s'}}\right)\big|_{J_s'} \\
&\quad= -\left(c_{b+1}\tau_{b+1} + \cdots + c_{b'}\tau_{b'} + d_{s,1}f_{s,1} + \cdots + d_{s,a_{s}} f_{s,a_{s}}\right)\big|_{J_s'}
\end{align*}
Let $g$ denote the functional on $J_{s'}$ that is given by each side of the above equation.
Notice that $g|_{J_s} = 0$ from the right-hand side and $g|_{J_{s'}}=0$ from the left-hand side.
Therefore, $g=0$.
From this, as in case 1, we can show that $d_{s,t}=0$ for all $t$.

Going back to the case that $I$ is not of the form $J_{s'}$, so that the induction hypothesis holds for $J_s$, in this situation, we can apply the same argument to $J_s'$ to show that $d_{s,t}=0$ for all $t$.
\end{enumerate}

Finally, knowing that $d_{s,t}=0$ for all $t$, when we now restrict \eqref{homogeqn} to $I$, we get
\[ \left(c_1\tau_1 + \cdots + c_k\tau_k\right)\big|_I = 0, \]
where for each $i$, either $\tau_i|_I=0$ or $c_i=0$ or $\tau_i|_I$ is a scalar multiple of an extreme state on $I$.
Since the extreme states of $I$ are linearly independent, it follows that $c_i=0$ whenever $\tau_i|_I \neq 0$.
\end{proof}

We are now set to prove Theorem \ref{RepresentationThm}.

\begin{proof}[Proof of Theorem \ref{RepresentationThm}]
This is achieved by relabelling the functionals
\[ \tau_1, \dots, \tau_k, f_{1,1},\dots,f_{1,a_1},\dots,f_{\ell,1},\dots,f_{\ell,a_\ell} \]
as $\phi_1,\dots,\phi_n$.
We already know that $\phi$ is one-to-one, and $\phi_1,\dots,\phi_n$ are linearly independent.

For each ideal $I$ of $G$, set
\[ S_I := \{i=1,\dots,n: \phi_i(I)\neq 0\}, \]
so that, by \eqref{J-intersectioneqn}, $I = I_{S_I}$.
Let $\mathcal{S}$ consist of each $S_I$ given by an ideal $I$.
Set
\[ P^\geq_{S_I} := \{i=1,\dots,n: \phi_i|_I \text{ is a positive functional}\}. \]
Then $P^\geq_{S_I} \supseteq S_I^c$, and $P^>_{S_I} = P^\geq_{S_I} \cap S_I$ consists of each $\tau_i$ that is a nonzero positive functional on $I$.

In light of the linear independence of $\phi_1,\dots,\phi_n$ and the choice of $\tau_i$'s, this means that the extreme states on $I$ are, up to positive scalar multiples, exactly the $\phi_i$ for which $i \in P^>_{S_I}$, and so (iv) follows by \cite[Theorem 1.4]{EffrosHandelmanShen}.
Item (ii) follows from (iv), since every positive element of $G$ is an order unit in the ideal it generates.
To see (v), note that the first part is a consequence of the fact that $\{ \phi_1,\dots,\phi_n \}$ is linearly independent, and the second part follows from the fact that the order units of the ideal $I$ separate the extreme states on $I$.

Finally, for two ideals $I$ and $J$ of $G$, we have
\[ S_{I+J} = \{i: \phi_i(I+J) \neq 0\} = \{i: \phi_i(I)\neq 0 \text{ or } \phi_i(J) \neq 0\} = S_I \cup S_J, \]
and
\[ S_{I \cap J} = \{i: \phi_i(I \cap J) \neq 0\} \subseteq S_I \cap S_J, \]
but also, since $I \cap J = I_{S_I} \cap I_{S_J}$ (and by (v)), we see that we must have $S_{I \cap J} \supseteq S_I \cap S_J$.
This shows that the map $I \to S_I$ is a lattice isomorphism, and that $\mathcal{S}$ is a lattice.
\end{proof}

\section{Necessary and sufficient conditions for interpolation}\alabel{Necessary-sufficient}

Theorem \ref{RepresentationThm} shows that, if a dimension group has finitely many ideals and finitely many extreme states on each ideal, and if, for each ideal $I$, $\degen{I}$ has countable rank, then the dimension group is isomorphic to a subgroup $G$ of $\R^n$, and the order can be described using some combinatorial data, which we shall give a name for now.

\begin{defn}
A \textbf{\triplename} is a triple $(\mathcal{S}, (P^>_S)_{S \in \mathcal{S}}, G)$ where $\mathcal{S}$ is a sublattice of $2^{\{1,\dots,n\}}$ containing $\emptyset$ and $\{ 1, \ldots , n\}$ , $P^>_S$ is a subset of $S$, and such that
\[
\{g \in G \ : \ g_i = 0 \ \forall i \not\in S\} \quad \textrm{spans} \quad  \{(x_1,\dots,x_n) \in \R^n \ : \ x_i = 0 \ \forall i \not\in S\}
\]
for each $S \in \mathcal{S}$.
\end{defn}

Associated to the \triplename $(\mathcal{S}, (P^>_S)_{S \in \mathcal{S}}, G)$ is the positive cone
\[
G^+ := \bigcup_{S \in \mathcal{S}} \{g \in G \ : \ g_i > 0 \ \forall i \in P^>_S, g_i=0 \ \forall i\not\in S\},
\]
making $G$ an ordered group.
Note that the cone on $(\mathcal{S}, (P^>_S)_{S \in \mathcal{S}}, \R^n)$ is
\[
\bigcup_{S \in \mathcal{S}} \{x \in \R^n \ : \ x_i > 0 \ \forall i \in P^>_S, x_i=0 \ \forall i\not\in S\},
\]
which is the cone generated by $G^+$.
Also, with this cone on $\R^n$, the inclusion $G \subseteq \R^n$ is an order embedding.

In this section, we will describe necessary and sufficient conditions on a \triplename \ $(\mathcal{S}, (P_S^>)_{S\in \mathcal{S}}, G)$ in order for $(G, G^+)$ to have Riesz interpolation 
The main result of this section is the following.

\begin{thm}\alabel{ConditionsThm}
Let $(\mathcal{S}, (P^>_S)_{S \in \mathcal{S}}, G)$ be a \triplename.  
Then $(G,G^+)$ is an unperforated ordered group, the ideals of which are
\[
I_S := \{ g\in G \ : \ g_i = 0 \ \forall i \notin S\}
\]
for $S\in\mathcal{S}$.

$(G, G^+)$ has Riesz interpolation if and only if the following conditions hold.
\begin{enumerate}
\item[(i)]  Letting $P_S^\geq = P_S^> \cup S^c$ for all $S\in \mathcal{S}$, we have, for $S_1,S_2 \in \mathcal{S}$,
\[ P^\geq_{S_1 \cup S_2} = P^\geq_{S_1} \cap P^\geq_{S_2}. \]
\item[(ii)]  For $S_1,S_2 \in \mathcal{S}$, if $S_1 \subsetneq S_2$ then
\[ P^>_{S_2} \not\subseteq S_1. \]
\item[(iii)]  For every pair $S_1,S_2 \in \mathcal{S}$, we have
\[ I_{S_1}+I_{S_2} = I_{S_1\cup S_2}. \]
\item[(iv)]  For every pair $S_1,S_2 \in \mathcal{S}$ for which $S_1$ is a maximal proper subset of $S_2$, we either have:
\begin{enumerate}
\item[(a)] If $S_2 \setminus S_1 = \{i_1,\dots,i_\ell\}$ then
\[ \{(x_{i_1},\dots,x_{i_\ell}): (x_1,\dots,x_n) \in I_{S_2}\} \]
is dense in $\R^{\ell}$.
\item[(b)] $S_2 = S_1 \dunion \{k\}$, and for any $T \in \mathcal{S}$ which contains $S_2$, either $k \not\in P^>_T$ or $T \setminus \{k\} \in \mathcal{S}$.
\end{enumerate}
\end{enumerate}
\end{thm}

\begin{proof}[Proof of the first part.]
Let us first prove that $(G, G^+)$ is an unperforated ordered group, the ideals of which are as described above.  
We can then discuss separately the necessity and sufficiency of conditions (i)--(iv).

That $G$ is an unperforated ordered group is clear from the definition of $G^+$.  

It is also clear from the definition of $G^+$ that for $S\in \mathcal{S}$ the set $I_S$ is convex.  

To see that $I_S$ is directed, let $g^{(1)}, \ldots , g^{(k)}\in I_S$ be an $\R$-basis for $\{(x_1,\dots,x_n) \in \R^n \ : \ x_i = 0 \ \forall i \not\in S\}$, and pick $h^{(1)}, h^{(2)} \in I_S$.  
Let $\| g^{(1)}\| ,\ldots , \| g^{(k)}\|$ denote the Euclidean norms of $g^{(1)}, \ldots , g^{(k)}$.  
Then we can find some $c\in\{(x_1,\dots,x_n) \in \R^n \ : \ x_i = 0 \ \forall i \not\in S\}$ such that, if $\| x-c\| < \|g^{(1)}\| + \cdots + \| g^{(k)}\|$, then $x_i > h_i^{(t)}$ for $t = 1, 2$ and for all $i\in P_S^>$.

The element $c$ can be written $c = c_1g^{(1)} + \cdots c_kg^{(k)}$ for $c_1, \ldots , c_k\in \R$.  
Taking nearest integers $\lfloor c_1\rfloor , \ldots , \lfloor c_k\rfloor$, we obtain an element
\[
h' := \lfloor c_1\rfloor g^{(1)} + \cdots + \lfloor c_k\rfloor g^{(k)} \in I_S
\]
such that $h_i' - h_i^{(t)} > 0$ for $t = 1,2$ and for all $i\in P_S^>$.

Therefore $h' - h^{(1)}, h'-h^{(2)} \in G^+$, and so $I_S$ is directed.

Finally, let us pick an arbitrary convex and directed subgroup $K$ of $G$ and show that it equals $I_S$ for some $S\in \mathcal{S}$.

Let $S$ be the largest set in $\mathcal{S}$ such that, for all $k\in K$, $k_i = 0$ for all $i\notin S$.  
Then certainly $K\subseteq I_S$.  

For each $i\in S$, there is some $k^{(i)}\in K$ such that $k_i^{(i)} \neq 0$.  
Then we can find some integer combination $k'$ of $\{ k^{(i)} \ : \ i\in S\}$ such that $k_i' \neq 0$ for $i\in S$.

Then $k' \in K$ because $K$ is a group, and because $K$ is directed, there is some $k^+\in K \cap G^+$ such that $k^+ \geq k'$.  
By considering the definition of $G^+$, we see that we must have $k_i^+ > 0$ for all $i\in P_S^>$.  
But then, because $K$ is convex, we see that $I_S \subseteq K$ as well.  
Therefore $K = I_S$ as required.
\end{proof}

Now let us consider the necessity of the conditions (i)--(iv).

The necessity of condition (iii) follows from \cite[Proposition 2.4]{Goodearl:book}, which says that, in a dimension group, the sum of two ideals is also an ideal.

The necessity of condition (i) then follows easily, because
\[ I_{S_1 \cup S_2}^+ = (I_{S_1}+I_{S_2})^+ = I_{S_1}^++I_{S_2}^+. \]

To show the necessity of condition (ii), suppose for a contradiction that $P^>_{S_2} \subseteq S_1$.
By (i), this implies that $P^>_{S_1} = P^>_{S_2}$, so that if $g$ is an order unit of $I_{S_1}$ then $g$ is an order unit of $I_{S_2}$.
This contradicts the hypothesis that $S_1 \subsetneq S_2$.

The next proposition, combined with the fact that the quotient of a dimension group by an ideal is once again a dimension group, shows that if condition (iv) (a) does not hold then $I_{S_2}/I_{S_1}$ must be cyclic.

\begin{prop}
Let $G$ be a simple dimension group with order unit and finitely many extreme states $\tau_1,\dots,\tau_k$.
Then either $(G,G^+) \iso (\mathbb{Z},\mathbb{N})$ or
\[ \{(\tau_1(g),\dots,\tau_k(g)): g \in G\} \]
is dense in $\R^k$.
\end{prop}

\begin{proof}
This is simply the combination of \cite[Proposition 14.3]{Goodearl:book} and \cite[Theorem 4.8]{GoodearlHandelman:Metric}.
\end{proof}

Now knowing that if condition (iv) (a) fails then $I_{S_2}/I_{S_1}$ is cyclic, the necessity of condition (iv) (b) comes from the following dichotomy, which generalizes \cite[Proposition 17.4]{Goodearl:book}.

\begin{prop}\alabel{CyclicDichotomy}
Let $G$ be a dimension group with an order unit and suppose that $I$ is a cyclic ideal of $G$.
If the unique state $\tau$ on $I$ extends to a positive functional on $G$ then $I$ is a direct summand of $G$, as an ordered group; that is, $G=I \dsum J$ for some ideal $J$.
\end{prop}

\begin{proof}
By Lemmas \ref{StateRestriction} and \ref{RestrictionOneToOne}, we see that there is a unique extreme state $\tau'$ on $G$, the restriction of which is a positive scalar multiple of $\tau$, and every other extreme state on $G$ is zero on $I$.

Let $h$ be the positive generator of $I$, and let $u$ be an order unit of $G$.
Then, by considering separately the case that $\tau'(u)$ is rational or irrational, we see that there exist $m > 0$ and $n \geq 0$ such that the element $g:=-mu+nh$ satisfies $\tau'(g) \in [0,\tau'(h)/2)$.
Let us show that
\[ \begin{array}{c}0 \\ g \end{array} \leq \begin{array}{c}h \\ h-g.\end{array} \]
Certainly, we already know that $0 \leq h$.
The other inequalities amount to showing that both $h-g$ and $h-2g$ are non-negative.
For any state $s$ on $G$, we have that either $s = \tau'$ or $s(I) = 0$.
In the first case,
\[ \tau'(h-g) > \tau'(h) - \tau'(h)/2 > 0, \]
and likewise,
\[ \tau'(h-2g) > \tau'(h) - 2\tau'(h)/2 > 0. \]
In the second case,
\[ s(h-g) = s(-g) = ms(u) > 0, \]
and likewise, $s(h-2g) > 0$.
By \cite[Theorem 1.4]{EffrosHandelmanShen}, we see that $h-g$ and $h-2g$ are in fact order units in $G$.

Since $G$ has Riesz interpolation, there must be some interpolant $z$ satisfying
\[ \begin{array}{c}0 \\ g \end{array} \leq z \leq \begin{array}{c}h \\ h-g.\end{array} \]
However, since $h$ is the generator of the ideal $I$ and $0 \leq z \leq h$, we must have either $z=0$ or $z=h$.
Since $g \leq z \leq h-g$, in either case, it follows that $g \leq 0$ and $\tau'(g)=0$.
Let $J$ be the ideal generated by $-g$; by the definition of $g$, we see that $G$ is the ideal generated by $I$ and $J$.
Yet, since $\tau'(-g)=0$, it follows that $h \not\in J$ and thus $I \cap J = \{0\}$.
Hence, $G=I\dsum J$ as required.
\end{proof}

Quotients of dimension groups by ideals are once again dimension groups, so the result of Proposition \ref{CyclicDichotomy} applies in particular to the cyclic subideal $I_{S\dunion \{ k\}}/I_S$ of $I_T/I_S$.  To say that the unique state on $I_{S\dunion \{ k\}}/I_S$ does not extend to $I_T/I_S$ means that $k\notin P_T^>$.  To say that $I_{S\dunion \{ k\}}/I_S$ is a direct summand in $I_T/I_S$ means that $T\setminus \{ k\} \in \mathcal{S}$ as well.

Therefore all of the conditions (i)--(iv) are necessary for $G$ to have Riesz interpolation.  
Let us now prove that they are also sufficient.

\begin{proof}[Proof of sufficiency of (i)--(iv).]
To show that $(G, G^+)$ has Riesz interpolation, choose elements $a^{(1)}, a^{(2)}, b^{(1)}, b^{(2)}\in G$ satisfying
\[
\begin{array}{l}
a^{(1)}\\
a^{(2)}
\end{array} \leq \begin{array}{l}
b^{(1)}\\
b^{(2)}
\end{array}.
\]

Say that $a^{(t)} = (a_1^{(t)}, \ldots , a_n^{(t)})$ and $b^{(t)} = (b_1^{(t)},\ldots , b_n^{(t)})$.

For each pair $s, t\in \{ 1, 2\}$, define
\[
T_{s,t} := \bigcap \{ S\in \mathcal{S} \ : \ a_i^{(s)} = b_i^{(t)} \ \forall i \notin S\}.
\]

Clearly, $T_{s,t} \in \mathcal{S}$ and $a_i^{(s)} = b_i^{(t)}$ for all $i\notin T_{s,t}$.

Furthermore, we must have $b_i^{(t)} - a_i^{(s)} > 0$ for all $i\in P_{T_{s,t}}^>$.  Indeed, to see this, suppose otherwise.  
Then, because $b^{(t)}-a^{(s)}\in G^+$, there is some set $S\in \mathcal{S}$ such that $a_i^{(s)} = b_i^{(t)}$ for all $i\notin S$ (so that $T_{s,t}\subseteq S$), and $b_i^{(t)} - a_i^{(s)} > 0$ for all $i\in P_S^>$ (so that $T_{s,t} \subsetneq S$).  
But, since $b_i^{(t)} - a_i^{(s)} = 0$ for all $i\in S\setminus T_{s,t}$, we must have $P_S^> \subseteq T_{s,t}$, contradicting condition (ii).

Next let us show that 
\begin{equation}
T_{1,1}\cup T_{2,2} = T_{1,2} \cup T_{2,1}. \alabel{Sufficiency-Teqn}
\end{equation}
Indeed, if we let $(b^{(1)}+b^{(2)}-a^{(1)}-a^{(2)})_i$ denote the $i$th entry of $b^{(1)}+b^{(2)}-a^{(1)}-a^{(2)}$ and $T$ denote the smallest set in $\mathcal{S}$ such that $(b^{(1)}+b^{(2)}-a^{(1)}-a^{(2)})_i = 0$ for all $i\notin T$, then both sides of \eqref{Sufficiency-Teqn} are equal to $T$.  
By symmetry, it suffices to show that  $T_{1,1}\cup T_{2,2} = T$.
 
We have $b_i^{(1)}-a_i^{(1)}=0$ for all $i\notin T_{1,1}$ and $b_i^{(2)}-a_i^{(2)}=0$ for all $i\notin T_{2,2}$, so $(b^{(1)}+b^{(2)}-a^{(1)}-a^{(2)})_i = (b_i^{(1)}-a_i^{(1)}) + (b_i^{(2)}-a_i^{(2)}) = 0$ for all $i\notin T_{1,1}\cup T_{2,2}$.  
Moreover, $T_{1,1}\cup T_{2,2}\in \mathcal{S}$, so by the minimality of $T$, we must have $T\subseteq T_{1,1}\cup T_{2,2}$.

Suppose for a contradiction that $T\subsetneq T_{1,1}\cup T_{2,2}$.  
Since $b_i^{(s)} - a_i^{(s)} > 0$ for $i \in P^>_{s,s}$, 
\[ (b^{(1)}+b^{(2)}-a^{(1)}-a^{(2)})_i > 0 \]
for $i \in (P^>_{T_{1,1}} \cap T_{2,2}^c) \cup (P_{T_{2,2}}^> \cap T_{1,1}^c) \cup (P_{T_{1,1}}^> \cap P^>_{T_{2,2}})$.  
But this set is exactly $(P_{T_{1,1}}^\geq \cap P_{T_{2,2}}^\geq )\setminus (T_{1,1}^c\cap T_{2,2}^c)$.  
By condition (i), $(P_{T_{1,1}}^\geq \cap P_{T_{2,2}}^\geq ) = P_{T_{1,1}\cup T_{2,2}}^\geq$, so $(b^{(1)}+b^{(2)}-a^{(1)}-a^{(2)})_i > 0$ for all $i\in P_{T_{1,1}\cup T_{2,2}}^\geq \setminus (T_{1,1}^c\cap T_{2,2}^c) = P_{T_{1,1}\cup T_{2,2}}^>$.  
But $(b^{(1)}+b^{(2)}-a^{(1)}-a^{(2)})_i = 0$ for all $i\notin T$, which means that $P_{T_{1,1}\cup T_{2,2}}^> \subseteq T$, and this contradicts condition (ii).  
Therefore $T = T_{1,1} \cup T_{2,2}$.  

For $i\notin (T_{1,1}\cap T_{1,2})\cup (T_{2,1}\cap T_{2,2})$, this means that for some $t, t'$ we have
\[
a_i^{(1)} = b_i^{(t)} \quad \textrm{and} \quad a_i^{(2)} = b_i^{(t')}.
\]

If $t = t'$, then $a_i^{(1)} = b_i^{(t)} = b_i^{(t')} = a_i^{(2)}$.

If $t \neq t'$, then $i\in (T_{1,t}\cup T_{2,t'})^c$, which, by \eqref{Sufficiency-Teqn}, is equal to $(T_{1,t'}\cup T_{2,t})^c$.  
Then either $a_i^{(1)} = b_i^{(t')}$ or $a_i^{(2)} = b_i^{(t)}$; in either case, we have $a_i^{(1)} = a_i^{(2)}$.  
Therefore $a^{(1)} - a^{(2)} \in I_{(T_{1,1}\cap T_{1,2})\cup (T_{2,1}\cap T_{2,2})}$.

By condition (iii), there exist $g^{(1)}\in I_{T_{1,1}\cap T_{1,2}}$ and $g^{(2)}\in I_{T_{2,1}\cap T_{2,2}}$ such that
\[
a^{(1)} - a^{(2)} = g^{(2)} - g^{(1)}.
\]
Set $c := a^{(1)} + g^{(1)} = a^{(2)} + g^{(2)}$, and let $c_i$ denote the $i$th entry of $c$.  The $i$th entry of $g^{(1)}$ is $0$ for $i\notin T_{1,1}\cap T_{1,2}$, so
\[
 c_i = a_i^{(1)}
\]
whenever $a_i^{(1)} = b_i^{(t)}$ for some $t$.  Likewise, we have $c_i = a_i^{(2)}$ whenever $a_i^{(2)} = b_i^{(t)}$ for some $t$.

Let $R$ denote the set $T_{1,1}\cap T_{1,2}\cap T_{2,1} \cap T_{2,2} \in \mathcal{S}$.  
Suppose that we are in the situation of condition (iv) (b), and $R$ plays the role of $T$; that is, there exist a set $S$ and index $k$ such that $S\dunion \{ k\} \in \mathcal{S}$, $S\dunion \{ k\} \subseteq R$, and $I_{S\dunion \{ k\}}/I_S$ is cyclic.  Suppose further that $k\in P_R^>$.  Then, by condition (iv) (b), $R\setminus \{ k\}$ is also in $\mathcal{S}$.

We can repeat this procedure now with the set $R_2 := R \setminus \{ k\}$.  
If there exist $S_2$ and $k_2 \notin S_2$ such that $S_2,S_2\dunion \{ k_2\} \in \mathcal{S}$, $S_2\dunion \{ k_2\} \subseteq R_2$, $I_{S_2\dunion \{ k_2\}}/I_{S_2}$ is cyclic, and $k_2 \in P_{R_2}^>$, then $R_3 := R_2\setminus \{ k_2\}$ is also in $\mathcal{S}$.

Continuing in this fashion, we will eventually arrive at a set $\hat{R} \in \mathcal{S}$ such that, if $S, S \dunion \{k\}$ fall into condition (iv) (b) and $S\dunion \{ k\} \subseteq \hat{R}$, then $k\notin P_{\hat{R}}^>$.  

Using an induction argument and condition (iv) (a), we can show that, if $P^>_{\hat{R}} = \{i_1,\dots,i_\ell\}$ then
\[ \{(g_{i_1},\dots,g_{i_\ell}): g \in I_{\hat{R}}\} \]
is dense in $\R^\ell$.  

Note that for $i \in P^>_{\hat{R}}$, we have
\[ \begin{array}{c} a_i^{(1)} \\ a_i^{(2)} \end{array} < \begin{array}{c} b_i^{(1)} \\ b_i^{(2)} \end{array} \]
Therefore we can find $m \in I_{\hat{R}}$ such that the element $z := c + m$ satisfies
\[
\max \{ a_i^{(1)}, a_i^{(2)}\} < z_i < \min \{ b_i^{(1)}, b_i^{(2)}\}
\]
for $i \in P^>_{\hat{R}}$.
By our choice of $c$, we have
\[
z_i = \max \{ a_i^{(1)}, a_i^{(2)}\}
\]
for all $i$ satisfying
\[
\max \{ a_i^{(1)}, a_i^{(2)}\} = \min \{ b_i^{(1)}, b_i^{(2)}\} \quad \textrm{ or} \quad i \in R\setminus \hat{R}
\]

Therefore, we see that for $t=1,2$ we have $z - a^{(t)} \in I_{\hat{R}}^+$ and $b^{(t)} - z \in I_{R}^+$.  Therefore
\[
\begin{array}{l}
a^{(1)}\\
a^{(2)}
\end{array} \leq z \leq \begin{array}{l}
b^{(1)}\\
b^{(2)}
\end{array},
\]
as required.
\end{proof}

\section{Classification of finite dimensional ordered real vector spaces with Riesz interpolation\alabel{Rvspaces}}

Certain of the conditions in Theorem \ref{ConditionsThm} are automatically satisfied in the case that the group $G$ is all of $\R^n$.
We have eliminated these conditions to give the following corollary.

\begin{cor}\alabel{RvspacesCor}
Let $(\mathcal{S}, (P_S)_{S\in\mathcal{S}}, \R^n)$ be a \triplename \ and use $V^+$ to denote the associated positive cone.

Then $(\R^n, V^+)$ is an ordered real vector space, and it has Riesz interpolation if and only if the following two conditions hold:
\begin{enumerate}
\item[(i)]  Letting $P_S^\geq = P_S^> \cup S^c$ for all $S\in \mathcal{S}$, we have, for $S_1,S_2 \in \mathcal{S}$,
\[ P^\geq_{S_1 \cup S_2} = P^\geq_{S_1} \cap P^\geq_{S_2}; \text{ and}\]
\item[(ii)]  For $S_1,S_2 \in \mathcal{S}$, if $S_1 \subsetneq S_2$ then
\[ P^>_{S_2} \not\subseteq S_1. \]
\end{enumerate}
\end{cor}

By Theorem \ref{RepresentationThm} and the following remark, every finite dimensional ordered real vector space with interpolation occurs as one in the Corollary just described, where $(\mathcal{S}, (P^>_S)_{S \in \mathcal{S}})$ is determined up to permutation of the indices $1,\dots,n$.
This yields the following classification.

\begin{cor}\alabel{Rvspaces-Classification}
For $n \in \mathbb{N}$, let $D_n$ denote the set of all pairs $(\mathcal{S},(P^\geq_S)_{S \in \mathcal{S}})$ where
\begin{enumerate}
\item $\mathcal{S}$ is a sublattice of $2^{\{1,\dots,n\}}$ that contains $\emptyset$ and $\{1,\dots,n\}$;
\item $P^\geq_S$ is a subset of $\{1,\dots,n\}$ that properly contains $S^c$;
\item For $S_1,S_2 \in \mathcal{S}$,
\[ P^\geq_{S_1 \cup S_2} = P^\geq_{S_1} \cap P^\geq_{S_2}; \text{ and} \]
\item For $S_1,S_2 \in \mathcal{S}$, if $S_1 \subsetneq S_2$ then
\[ P^>_{S_2} \not\subseteq S_1. \]
\end{enumerate}
Define the equivalence relation $\equiv$ on $D_n$ by $(\mathcal{S},(P^\geq_S)_{S \in \mathcal{S}}) \equiv (\mathcal{S}',(Q^\geq_S)_{S \in \mathcal{S}'})$ if there exists a permutation $\sigma$ on $\{1,\dots,n\}$ such that
\[ \mathcal{S}' = \{\sigma(S): S \in \mathcal{S}\}, \]
and for each $S \in \mathcal{S}$,
\[ Q^\geq_{\sigma(S)} = \sigma(P^\geq_S). \]
Then the (ordered real vector space) isomorphism classes of $n$ dimensional ordered real vector spaces with Riesz interpolation are in one-to-one correspondence with the equivalence classes in $D_n$.
\end{cor}

\begin{remark}
If $V,V'$ are two finite dimensional ordered real vector spaces corresponding to two elements $d,d'$ of $D_n$, and $V \iso V'$ then the isomorphism gives rise to a permutation $\sigma$ that induces the equivalence $d \equiv d'$.
While a permutation $\sigma$ that induces an equivalence $d \equiv d'$ lifts to an isomorphism $V \to V'$, this isomorphism need not be unique since there is choice in the maps $\degen{I_{S}} \to \degen{I_{\sigma(S)}}$.
\end{remark}

\section{Classification of finite rank dimension groups}\alabel{dgroups}

In light of Theorem \ref{ConditionsThm}, we see that for any dimension group $G$ satisfying the hypotheses of Theorem \ref{RepresentationThm} and any representation $\phi:G \to \R^n$ given by that theorem, if $V^+$ denotes the real cone in $\R^n$ generated by $\phi(G^+)$ (that is, $V^+$ is closed under addition and multiplication by positive scalars) then $(\R^n,V^+)$ has Riesz interpolation.
In particular, we have the following.

\begin{cor}\alabel{Dimension-subgroup}
Let $G$ be a dimension group that has finitely many ideals, finitely many extreme states on each ideal, and such that, for every ideal $I$, the rank of $\degen{I}$ is at most $2^{\aleph_0}$.
Then $G$ arises as a subgroup of an ordered real vector space with interpolation (with the induced ordering).
In fact, the ordered real vector space can be chosen to have finite dimension.
\end{cor}

If $\degen{I}$ has finite rank for every ideal $I$ (in particular, if $G$ has finite rank) then, according to the remark following Theorem \ref{RepresentationThm}, we can find a canonical representation, i.e.\ we can impose a restriction on the representation $\phi$ that ensures that two embeddings $\phi$ and $\phi'$ of $G$ into $\R^n$ give rise to an isomorphism of $\R^n$ (in fact, it is an isomorphism of ordered real vector spaces, when we equip $\R^n$ with the cone generated by $\phi(G^+)$ and that generated by $\phi'(G^+)$).
Here are the conditions on the embedding $\phi:G \to \R^n$ that make it canonical; throughout, we let $V^+$ denote the real cone generated by $\phi(G^+)$, and use $V$ to denote the ordered real vector space $\R^n$, the positive cone of which is $V^+$.
\begin{enumerate}
\item[(i)] The ordered real vector space $(V,V^+)$ has Riesz interpolation;
\item[(ii)] The ideals of $G$ correspond to the ideals of $V$ via the map $I \mapsto \Span \phi(I)$; and
\item[(iii)] For every pair of ideals $I$ and $J$, we have
\[ \Span \phi(I+J) = \Span \phi(I) + \Span \phi(J). \]
\item[(iv)] For each ideal $I$ of $G$, the rank of $\degen{I}$ equals the dimension of $\degen{\phi(I)}$.
\end{enumerate}
A consequence of these conditions (which is explicit in the proof of Theorem \ref{RepresentationThm}) is that the positive functionals on each ideal $I$ correspond exactly to the positive functionals on $\phi(I)$, so that if we pick an order unit for $I$ and use its image as an order unit for $\phi(I)$ then this produces an identification of the state space of $I$ and of $\phi(I)$.

\begin{cor}\alabel{dgroups-Classification}
Let $C$ denote the set of all \triplename s $(\mathcal{S},(P^>_S)_{S \in \mathcal{S}}, G)$ that satisfy conditions (i)-(iv) of Theorem \ref{ConditionsThm} as well as:
\begin{enumerate}
\item[(v)] For every $S \in \mathcal{S}$, letting 
\[ \{i_1,\dots,i_k\} = \{i=1,\dots,n: i \not\in S \text{ and } i \not\in P^>_T\ \forall T \supset S\}, \]
then the rank of
\[ \{(x_{i_1},\dots,x_{i_k}): (x_1,\dots,x_n) \in G \text{ and } x_i = 0\ \forall i \not\in \{i_1,\dots,i_k\} \cup S\} \]
(as an abelian group, i.e., regardless of the embedding into $\R^k$) is equal to $k$.
\end{enumerate}
Define the equivalence relation $\equiv$ on $C$ by $(\mathcal{S},(P^>_S)_{S \in \mathcal{S}}, G) \equiv (\mathcal{S}',(Q^>_S)_{S \in \mathcal{S}'}, H)$ if $\mathcal{S}$ and $\mathcal{S'}$ are sublattices of $2^{\{1,\dots,n\}}$ for the same $n$, and there exists a permutation $\sigma$ on $\{1,\dots,n\}$ and a vector space isomorphism $\phi:\R^n \to \R^n$ satisfying the following.
\begin{enumerate}
\item[(i)] $\mathcal{S}' = \{\sigma(S): S \in \mathcal{S}\}$;
\item[(ii)] For each $S \in \mathcal{S}$, $Q^>_{\sigma(S)} = \sigma(P^>_S)$; and
\item[(iii)] For each $S \in \mathcal{S}$,
\begin{align*}
\phi(\{(x_1,\dots,x_n) \in \R^n: x_i = 0\ \forall i \not\in S \text{ and } x_i > 0\ \forall i \in P^>_S\}) \\
= \{(x_1,\dots,x_n) \in \R^n: x_i = 0\ \forall i \not\in \sigma(S) \text{ and } x_i > 0\ \forall i \in \sigma(P^>_S)\}; \text{ and}
\end{align*}
\item[(iv)] $\phi(G) = H'$.
\end{enumerate}
Then the equivalence classes in $C$ are in one-to-one correspondence with the (ordered group) isomorphism classes of dimension groups which have finitely many ideals, finitely many extreme traces on each ideal, and for which $\degen{I}$ has finite rank for each $I$.
\end{cor}

\begin{remark}
The classification of finite rank dimension groups follows: these correspond exactly to the equivalence classes in $C$ for which the group $G$ has finite rank.
\end{remark}

\begin{proof}
This is an immediate consequence of Theorems \ref{RepresentationThm} and \ref{ConditionsThm}, and the remark following Theorem \ref{RepresentationThm}.
\end{proof}

\begin{cor}\alabel{Qtensor-characterization}
Let $(G,G^+)$ be a finite rank ordered group that is unperforated.
Then $G$ is a dimension group if and only if $G \tens \Q$ is a dimension group and, for every triple $I,J,K$ of ideals of $G$, such that $I \subset J \subset K$ and $J/I$ is cyclic, either the unique state on $J/I$ does not extend to $K/I$ or else $J/I$ is a direct summand in $K/I$.
\end{cor}

\begin{question}
Although our methods only allow us to prove Corollaries \ref{Dimension-subgroup} and \ref{Qtensor-characterization} for dimension groups with finite-rank type restrictions, there is no obvious reason why these results would not hold for general dimension groups (the generalization of Corollary \ref{Dimension-subgroup} would not use a finite dimensional ordered real vector space).
This leads one to ask: do these results of them hold for general dimension groups?
\end{question}


\begin{thebibliography}{10}

\bibitem{CuntzKrieger:algebras}
Joachim Cuntz and Wolfgang Krieger.
\newblock A class of {$C^{\ast} $}-algebras and topological {M}arkov chains.
\newblock {\em Invent. Math.}, 56(3):251--268, 1980.

\bibitem{Edwards}
David~Albert Edwards.
\newblock S\'eparation des fonctions r\'eelles d\'efinies sur un simplexe de
  {C}hoquet.
\newblock {\em C. R. Acad. Sci. Paris}, 261:2798--2800, 1965.

\bibitem{EffrosHandelmanShen}
Edward~G. Effros, David~E. Handelman, and Chao~Liang Shen.
\newblock Dimension groups and their affine representations.
\newblock {\em Amer. J. Math.}, 102(2):385--407, 1980.

\bibitem{EffrosShen:DimGroups}
Edward~G. Effros and Chao~Liang Shen.
\newblock Dimension groups and finite difference equations.
\newblock {\em J. Operator Theory}, 2(2):215--231, 1979.

\bibitem{EffrosShen:ContinuedFractions}
Edward~G. Effros and Chao~Liang Shen.
\newblock Approximately finite {$C^{\ast} $}-algebras and continued fractions.
\newblock {\em Indiana Univ. Math. J.}, 29(2):191--204, 1980.

\bibitem{EffrosShen:Geometry}
Edward~G. Effros and Chao~Liang Shen.
\newblock The geometry of finite rank dimension groups.
\newblock {\em Illinois J. Math.}, 25(1):27--38, 1981.

\bibitem{Elliott:classificationAF}
George~A. Elliott.
\newblock On the classification of inductive limits of sequences of semisimple
  finite-dimensional algebras.
\newblock {\em J. Algebra}, 38(1):29--44, 1976.

\bibitem{Elliott:classificationRR0}
George~A. Elliott.
\newblock On the classification of {$C^*$}-algebras of real rank zero.
\newblock {\em J. Reine Angew. Math.}, 443:179--219, 1993.

\bibitem{Elliott:classificationAT}
George~A. Elliott.
\newblock A classification of certain simple {$C^*$}-algebras. {II}.
\newblock {\em J. Ramanujan Math. Soc.}, 12(1):97--134, 1997.

\bibitem{Farhane:coding}
Lotfi Farhane.
\newblock Coding of the dimension group.
\newblock {\em Adv. Math.}, 206(2):455--465, 2006.

\bibitem{Fuchs:InterpVSpaces}
L.~Fuchs.
\newblock On partially ordered vector spaces with the {R}iesz interpolation
  property.
\newblock {\em Publ. Math. Debrecen}, 12:335--343, 1965.

\bibitem{Fuchs:RieszGrps}
L.~Fuchs.
\newblock Riesz groups.
\newblock {\em Ann. Scuola Norm. Sup. Pisa (3)}, 19:1--34, 1965.

\bibitem{Fuchs:rieszspaces}
L.~Fuchs.
\newblock {\em Riesz vector spaces and {R}iesz algebras}.
\newblock Queen's Papers in Pure and Applied Mathematics, No. 1. Queen's
  University, Kingston, Ont., 1966.

\bibitem{Goodearl:book}
K.~R. Goodearl.
\newblock {\em Partially ordered abelian groups with interpolation}, volume~20
  of {\em Mathematical Surveys and Monographs}.
\newblock American Mathematical Society, Providence, RI, 1986.

\bibitem{GoodearlHandelman:Metric}
K.~R. Goodearl and D.~E. Handelman.
\newblock Metric completions of partially ordered abelian groups.
\newblock {\em Indiana Univ. Math. J.}, 29(6):861--895, 1980.

\bibitem{GoodearlHandelmanLawrence}
K.~R. Goodearl, D.~E. Handelman, and J.~W. Lawrence.
\newblock Affine representations of {G}rothendieck groups and applications to
  {R}ickart {$C^{\ast} $}-algebras and {$\aleph _{0}$}-continuous regular
  rings.
\newblock {\em Mem. Amer. Math. Soc.}, 26(234):vii+163, 1980.

\bibitem{Krieger:DimFns}
Wolfgang Krieger.
\newblock On dimension functions and topological {M}arkov chains.
\newblock {\em Invent. Math.}, 56(3):239--250, 1980.

\bibitem{Villadsen:AHrange}
Jesper Villadsen.
\newblock The range of the {E}lliott invariant of the simple {AH}-algebras with
  slow dimension growth.
\newblock {\em $K$-Theory}, 15(1):1--12, 1998.

\end{thebibliography}
\end{document}